\documentclass[reqno,11pt]{amsart}
 \usepackage{amsaddr}
\usepackage[text={500pt,660pt},centering]{geometry}
\usepackage{tikz}
\usetikzlibrary{positioning}
\usetikzlibrary{shapes.geometric, arrows}
\usepackage{color}
\usepackage{amssymb}
\usepackage{graphicx}
\usepackage{MnSymbol}
\usepackage{mathtools}
\usepackage[colorlinks=true, pdfstartview=FitV, linkcolor=blue, citecolor=blue, urlcolor=blue,pagebackref=false]{hyperref}
\usepackage{microtype}

\usepackage{bm}
\usepackage{scalerel} 
\usepackage{dsfont}
\usepackage{mathrsfs}
\usepackage{eufrak}
\usepackage[font={footnotesize}]{caption}
\usepackage[shortlabels]{enumitem}

\definecolor{darkgreen}{rgb}{0,0.5,0}
\definecolor{darkblue}{rgb}{0,0,0.7}
\definecolor{pink}{rgb}{0.8,0.25,0.5125}
\definecolor{purple}{rgb}{0.5,0.3,0.9}

\tikzstyle{startstop} = [rectangle, rounded corners, minimum width=3cm, minimum height=1cm,text centered, draw=black, fill=red!30]
 \tikzstyle{io} = [trapezium, trapezium left angle=70, trapezium right angle=110, minimum width=3cm, minimum height=1cm, text centered, draw=black, fill=blue!30]
 \tikzstyle{arrow} = [thick,->,>=stealth]

\tikzstyle{process} = [rectangle, minimum width=3cm, minimum height=1cm, text centered, draw=black, fill=orange!30]

\newtheorem{theorem}{Theorem}[section]

\newtheorem{lemma}[theorem]{Lemma}
\newtheorem{corollary}[theorem]{Corollary}

\theoremstyle{remark}
\newtheorem{remark}[theorem]{Remark}

\theoremstyle{definition}

\numberwithin{equation}{section}
\numberwithin{equation}{section}

\newcommand{\fint}{\strokedint}

\newcommand{\zmh}{\underline{L}^2}

\newcommand{\RR}{\mathbb{R}}

\newcommand{\e}{\varepsilon}

\newcommand{\de}{\delta}

\newcommand{\x}{\xi}

\newcommand{\rst}[1]{\ensuremath{{\mathbin\upharpoonright}%
\raise-.5ex\hbox{$#1$}}}

\newcommand{\sh}{\mathcal{H}}

\newcounter{gscan}
\newcounter{bwscan}
\newcounter{cscan}
\newcounter{hscan}
\newcounter{fscan}
\newcounter{pscan}
\newcounter{sscan}
\newcounter{iscan}
\newcounter{rscan}
\newcounter{rrscan}
\newcounter{fpscan}

\renewcommand{\tilde}{\widetilde}

\newcommand{\tstar}[5]{
\pgfmathsetmacro{\starangle}{360/#3}
\draw[#5] (#4:#1)
\foreach \x in {1,...,#3}
{ -- (#4+\x*\starangle-\starangle/2:#2) -- (#4+\x*\starangle:#1)
}
-- cycle;
}

\newcommand{\ngram}[4]{
\pgfmathsetmacro{\starangle}{360/#2}
\pgfmathsetmacro{\innerradius}{#1*sin(90-\starangle)/sin(90+\starangle/2)}
\tstar{\innerradius}{#1}{#2}{#3}{#4}
}

\usepackage[utf8]{inputenc}

\title[Complex Analytic Dependence in ENZ Materials]{Complex Analytic Dependence on the Dielectric Permittivity in ENZ Materials: The Photonic Doping Example}
\author{Robert V. Kohn \and Raghavendra Venkatraman}
\address{Courant Institute of Mathematical Sciences}
\email{kohn@cims.nyu.edu,raghav@cims.nyu.edu}
\date{\today}

\begin{document}
\begin{abstract}
Motivated by the physics literature on ``photonic doping'' of scatterers made from ``epsilon-near-zero'' (ENZ) materials, we consider how the scattering of time-harmonic TM electromagnetic waves by a cylindrical ENZ region $\Omega \times \RR$ is affected by the presence of a ``dopant'' $D \subset \Omega$ in which the dielectric permittivity is not near zero. Mathematically, this reduces to analysis of a 2D Helmholtz equation $\mathrm{div}\, (a(x)\nabla u) + k^2 u = f$ with a piecewise-constant, complex valued coefficient $a$ that is nearly infinite
(say $a = \frac{1}{\delta}$ with $\delta \approx 0$) in $\Omega \setminus \overline{D}.$ We show (under suitable hypotheses) that the solution $u$ depends analytically on $\delta$ near $0$, and we give a simple PDE characterization of the terms in its Taylor expansion. For the application to photonic doping, it is the leading-order corrections in $\delta$ that are most interesting: they explain why photonic doping is only mildly affected by the presence of losses, and why it is seen even at frequencies where the dielectric permittivity is merely small. Equally important: our results include a PDE characterization of the leading-order electric field in the ENZ region as $\delta \to 0$, whereas the existing literature on photonic doping provides only the leading-order magnetic field.
\end{abstract}
\maketitle

 \section{Introduction}
 A body of literature has developed concerning the design of electromagnetic devices using epsilon-near-zero (ENZ) materials \cite{Eng-pursuing,Eng1,ENZ1}. A significant part of it focuses on the transverse magnetic setting, where the magnetic field has the form $\vec H(x) = (0,0,H(x_1,x_2))$ and the time-harmonic Maxwell's equations reduce  to the two-dimensional equation
 \begin{align} \label{e.helmholtz.intro}
     -\nabla \cdot \left( \frac{1}{\e} \nabla H \right) - \omega^2 \mu H = f.
 \end{align}
 Here $\e = \e(x_1,x_2)$ is the dielectric permittivity and $\mu = \mu(x_1,x_2)$ is the magnetic permeability, both of which are in general complex-valued and frequency-dependent (i.e., $\e = \e(\omega)$ and $\mu = \mu(\omega)$, where $\omega$ denotes frequency). Within this time-harmonic setting, the electric field is described by the gradient of the scalar function $H$ through $\vec E = \frac{1}{i \omega \e}(-\partial_{x_2} H, \partial_{x_1}H, 0).$ The function $f$ models a source that can be quite general.

 It is easy to see from \eqref{e.helmholtz.intro} what is special about ENZ materials: since it is natural to expect that $\frac{1}{\e}\nabla H$ is bounded, we expect $H$ to be nearly constant in a region where $\e$ is nearly $0$. In the limit $\e \rightarrow 0$, we are not solving a PDE in the ENZ region but rather selecting the constant value of $H$. This creates curious effects; for example, parts of the ENZ region that are distant from one another are nevertheless tightly coupled, since they share the same (constant) value of $H$; moreover the shape of the ENZ region is less important than it would be for a PDE, since all that is happening there is the selection of a constant $H$. This effect has for example been used (i) to design entirely new types of waveguides \cite{alu-etal-pre2008,silveirinha-engheta-prl2006,silveirinha-engheta-prb2007a,silveirinha-engheta-prb2007b,zhou-etal-pra2020}, and (ii) to show how ``dopants'' can give an ENZ region an ``effective magnetic permeability'' $\mu_{\mathrm{eff}}$ that is entirely different from its true magnetic permeability $\mu$ \cite{Eng2,silveirinha-engheta-prb2007a,zhou-etal-pra2020}.

 The preceding discussion assumes that the PDE \eqref{e.helmholtz.intro} has a meaningful limit as $\e \rightarrow 0$ in the region occupied by the ENZ material. We shall prove this, but we are mainly interested in the corrections that appear when $\e$ is \emph{nonzero}. Making a choice, we shall focus on the second application mentioned above-- photonic doping-- though the strategy we follow should also be applicable to waveguides.

 To state the main results of this paper, we begin by describing the photonic doping problem that we study. The scatterer has cross-section $\Omega \subset \RR^2$, with finitely many cylindrical dopant rods. We let $D \subset \Omega$ denote the union of the dopants (see Fig. \ref{f.intro}), and consider the Helmholtz equation \eqref{e.helmholtz.intro} with piecewise constant $\e $ and $\mu$, with $\e$ small in magnitude in the ENZ region $\Omega \setminus \overline{D}$. Our analysis is sufficiently general to permit $\e$ and $\mu$ taking distinct constant complex values in the exterior $\RR^2 \setminus \overline{\Omega},$ the ENZ region $\Omega \setminus \overline{D},$ and each of the dopant rods, respectively. So long as the $|\e| \ll 1$ in the ENZ region, the mathematical analysis of the problem is not complicated by the value of these other constants, as well as the (finite) number of connected components of the dopant $D$.
 \begin{figure}[h]
     \centering
     \includegraphics[scale=.6]{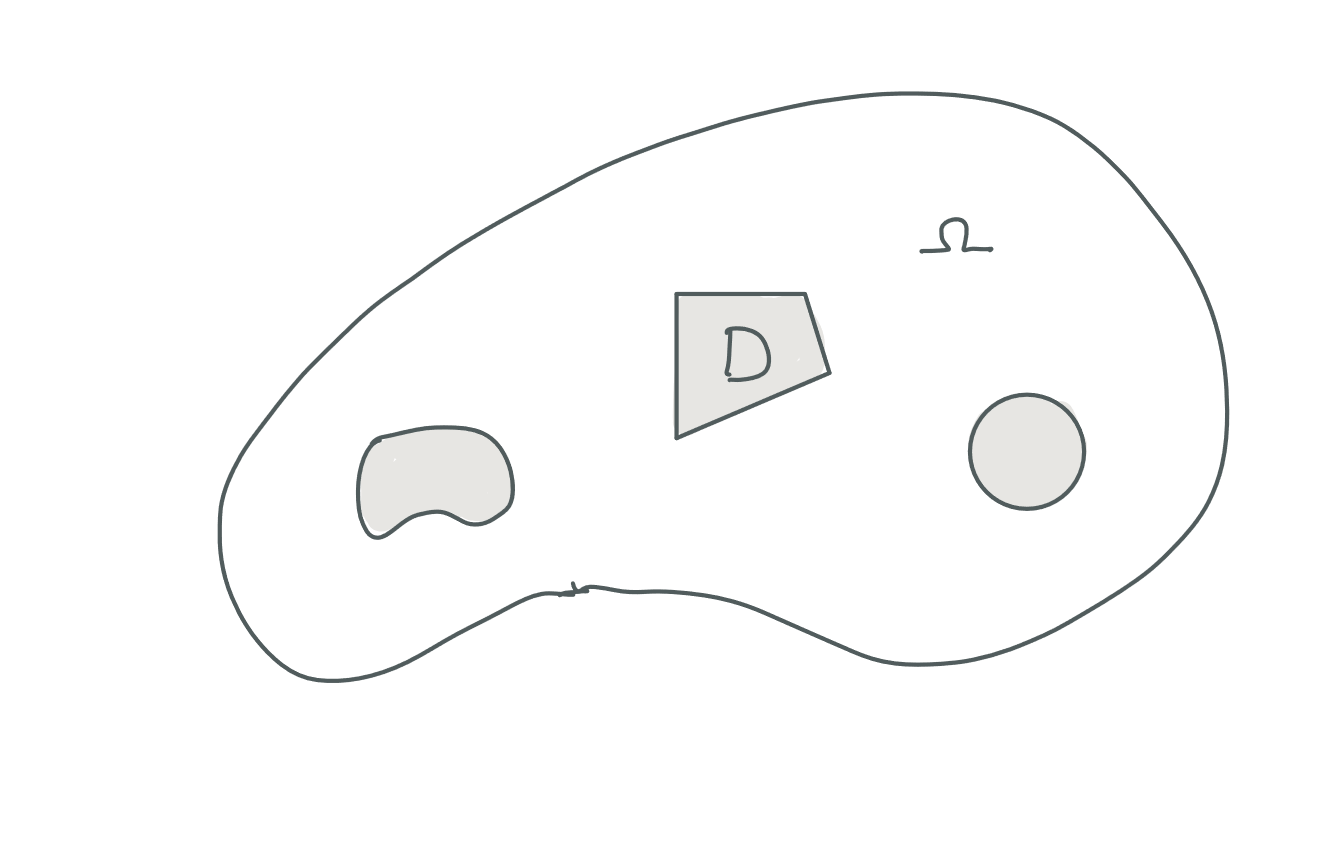}
     \caption{The photonic doping problem}
     \label{f.intro}
 \end{figure}
However, for simplicity, and in order to focus on the mathematical essence, we make some reductions.  We assume that there is a single dopant rod $D$ (see Fig. \ref{fig:setup}) immersed in the ENZ region; we take $\e = 1$ in the exterior; and we take $\e = \delta \in \mathbb{C} \setminus \{0\}, |\delta| \ll 1$ in the ENZ region. Also, to simplify the presentation, we assume that $\e = 1$ in the dopant. Thus, we shall study the case
\begin{align} \label{e.permi}
    \e_\delta (x) := \left\{
    \begin{array}{cc}
        1  & x \in \RR^2 \setminus \overline{\Omega},  \\
        \delta &  x \in \Omega \setminus \overline{D},\\
        1 & x \in D.
    \end{array}
    \right.
\end{align}
We shall assume that the dopant is not resonant (see Section \ref{s.asspt} for a precise statement, and
Section \ref{s.exam} for some discussion about what happens if the dopant is resonant). Our analysis requires no other condition on $\delta;$ in particular, it applies equally to the case $\mathrm{Im}(\delta) > 0$ (corresponding to a lossy material) and the case $\mathrm{Im}(\de) < 0$ (corresponding to a material with gain) -- provided, of course, that $|\delta|$ is sufficiently small. As for the value of $\mu$, we take it to be the same in all regions (though the case where it takes different values in the three regions would not be fundamentally different). For the exact assumptions on the regularity of $\Omega$ and $D$, as well as our assumptions on the source, we refer the reader to Section \ref{s.pbmr}.

Our main result, stated in Theorem \ref{t.mt} below, asserts that the solution to
\eqref{e.helmholtz.intro} depends analytically on the complex variable $\delta$ in a neighborhood of $0$.
Let us emphasize that the behavior of the magnetic field $\vec{H}$ in the limit $\delta = 0$ is correctly
found (though not proved) in the physics literature \cite{Eng2,silveirinha-engheta-prb2007a}.
As previously mentioned, this amounts to selecting the correct constant value of the scalar function $H$.
However, the physics literature offers no analytical understanding about:
\begin{enumerate}
    \item the character and magnitude of corrections due to $\delta \neq 0,$ and
    \item the character and magnitude of the electric field $\vec{E}$, which, as we show below, has a nonzero limit as $|\delta| \to 0.$
\end{enumerate}
As corollaries to our analysis, we obtain simple and reasonably explicit information on both points. Such
information is physically relevant. To explain why, we remind the reader that for real materials, the permittivity
$\e = \e' + i \e''$ is a complex-analytic function of the frequency $\omega$; moreover its imaginary part
represents losses (which are always present in real materials), so $\e''(\omega) > 0$ when
$\omega$ is real and positive (see e.g. \cite{landau-lifshitz} sections 77 and 80). Thus, while the notion
of an epsilon-near-zero material is a useful idealization, the physically relevant value of our
parameter $\delta$ has positive imaginary part, and the real part of $\delta$ is nonzero except at isolated frequencies.
For the effect of ``photonic doping'' to be robust with respect to material losses and frequency
variation, the corrections due to nonzero $\delta$ must be small. Our analyticity results show that
these corrections are of order $O(|\delta|)$.

Concerning (2), we recall that the Poynting vector $S$ is defined by
$S := \frac{1}{2} \vec{E} \times \overline{\vec{H}}$. (Here, as usual, for any complex number
$z,$ $\overline{z}$ denotes its complex conjugate.) The real part of $S$ is the time-averaged power flow,
while the imaginary part is the reactive power flow. It is of interest to understand the
spatial structure of these flows. Since the Poynting vector involves both the magnetic field $\vec{H}$
and the electric field $\vec{E},$ this requires understanding both fields in the limit $\delta \rightarrow 0$.
For ENZ waveguides, this was the focus of \cite{liberal2020near}; we are not aware of a similar study in the photonic doping setting.

\subsection*{Some related literature} \label{ss.mathlit}
We close this introduction with a brief discussion of some related literature.

\begin{enumerate}
\item The paper \cite{Bruno} by O. Bruno considers the PDE $\mathrm{div}(a(x)\nabla u) = 0$ for coefficient fields $a$ that take two values: the conductivity of the inclusions is $\delta \in \mathbb{C}\setminus \{0\},$ and that of the background matrix is $1.$ It studies the effective conductivity as a function of $\delta$ for both periodic and random coefficient fields $a,$ and provides conditions on the geometry of the inclusions that ensure analyticity of the effective conductivity near $\delta \approx 0$ and $\delta \approx \infty.$ Our approach is not fundamentally different, but our geometry is more specific; moreover, our focus is the Helmholtz equation.

\item The papers \cite{BF} and \cite{zhi} study how a periodic array of inclusions can lead to a frequency-dependent effective $\mu$, in a suitable homogenized limit. In photonic doping, an effective permeability $\mu_{\mathrm{eff}}$ emerges due to the presence of the dopant (see Remark \ref{r.mueff} below). However, our setting is different from \cite{BF,zhi} since we do not consider an array of inclusions, and in particular, no homogenization procedure is being carried out. And yet our setting shares some features with that of \cite{BF,zhi}, since here and also there, the effective permeability is directly related to the solution of a Helmholtz equation in the dopant with boundary condition $1.$

\item  In the present work, we assume that the ENZ region is a uniform electromagnetic medium,
characterized by two constants, its complex permittivity and permeability. There are real materials
and experimental systems for which this is the case
(see for example \cite{App4,edwards-etal-prl2008,ENZ1}).
A different but related body of literature is about metamaterials that achieve an effective $\e$
that is close to zero at certain frequencies, as a consequence of their specific
microstructure (see \cite{ENZ1}, or for a recent example \cite{lusk}). In such systems, the
robustness of the effect to the variation of the frequency $\omega$ or the presence of loss
depends not only on the effects studied here, but also on how well their behavior is captured
by the effective $\e$.

\item In our study of photonic doping the parameter $\delta$ tends to zero while the geometry
is fixed. It is of course interesting to ask whether other exotic
effects can be achieved by letting the microstructure vary as a physical parameter
approaches $0$ or $\infty$. The answer is yes; indeed, many papers on
electromagnetic metamaterials have this character.
For divergence-form equations of the form $\mathrm{div}\,(a(x)\nabla u) = 0$ with $u$ scalar, the full set of limits
obtainable this way has been characterized \cite{camar-eddine-seppecher-scalar}. The analogous question for
linear elasticity has also been addressed, and the set of all possible limits is
surprisingly rich \cite{camar-eddine-seppecher-elastic}. We are not aware of analogous results for
Maxwell's equations.

\end{enumerate}
\subsection*{Organization of the paper} After precisely formulating the problem and recording some preliminaries, we state our main results in Section \ref{s.pbmr}. This section also contains a roadmap of the proof of the main theorem (see subsection \ref{ss.mainidea}). We introduce certain operators that arise in describing higher order correctors in Section \ref{s.opt}. The proof of the main theorem and its corollaries are in Section \ref{s.proofs}, and we conclude the paper in Section \ref{s.exam} by considering what happens when the dopant is resonant (or nearly so).

\section{Formulation of the Problem and Main Results} \label{s.pbmr}

\subsection{Setting up the problem}
We are given a bounded domain $\Omega \subset \RR^2,$ with Lipschitz boundary (i.e., such that the boundary is locally the graph of a Lipschitz continuous function), and an open, Lipschitz subdomain $D \Subset \Omega$ that represents the dopant. Here and throughout the paper, we use the notation $A \Subset B$ to denote that
$\overline{A} \subset B$, where $\overline{A}$ is the closure of $A$. The ENZ region fills the set $\Omega \setminus \overline{D}.$ The unit outward pointing normal to $\partial \Omega$ and $\partial D$ are denoted by $\nu_\Omega$ and $\nu_D$ respectively.

We let $\delta \in \mathbb{C} \setminus \{0\}$ be a complex number, and define the spatially inhomogeneous but isotropic permittivity function as in \eqref{e.permi}. Suppose $f$ is a source that is supported compactly outside $\overline{\Omega}.$ To be precise, we suppose
\begin{align*}
    f \in L^2(K)
\end{align*}
for some compact set $K \Subset \RR^2 \setminus \overline{\Omega}$ (see Figure \ref{fig:setup}).
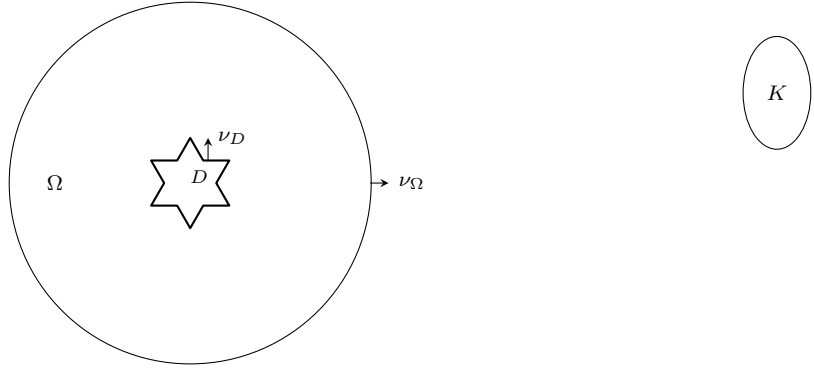
\begin{figure}[h]
    \centering
      \begin{tikzpicture}[line cap=round,line join=round,scale=.6,font=\scriptsize] \draw (0,0) circle (4.01cm) ;
        \draw (13,2) ellipse (.75 and 1.25) node []{$K$};
    \node at (-3, 0)  {$\Omega$};
\ngram{1}{6}{0}{thick} (.1cm); node  {\tiny $D$}    ;
   \draw [-stealth](3.6,0) -- (4,0) node [right]{$\nu_\Omega$};
   \draw [-stealth](0,.5) -- (0,1) node [right]{$\nu_D$};
    \end{tikzpicture}
    \caption{Set up of the problem: $D$ represents the dopant rod immersed in the ENZ material, which occupies the open set $\Omega \setminus \overline{D}$. The scattering response of the system is excited through a source $f$ supported in a compact set $K \subset \RR^2 \setminus \overline{\Omega}$. The unit outward pointing normal to the sets $D$ and $\Omega$ are denoted respectively by $\nu_D$ and $\nu_\Omega.$}
    \label{fig:setup}
\end{figure}
\noindent Finally, defining the operator
\begin{align*}
    \mathcal{L}_\delta := - \nabla \cdot \left( \frac{1}{\e_\delta(x)} \nabla  \right) - k^2 ,
\end{align*}
where $k$ is a nonzero complex number with $0 \leqslant \mathrm{arg} \, k < \pi$,
we consider the problem
\begin{eqnarray} \label{e.PDE}
    \mathcal{L}_\delta u_\delta = f ,  \mbox{ in } \RR^2, \\ \label{e.bcinfty}
    \lim_{r \to \infty} \sqrt{r}\left( \frac{\partial }{\partial r} - ik\right)u_\delta = 0.
\end{eqnarray}
Here, as usual, $r = |x|$ is the distance to the origin. The second condition is the Sommerfeld radiation condition, which assures uniqueness by selecting a solution that is ``outgoing'' at infinity. One interprets the PDE
in \eqref{e.PDE} along $\partial D$ and $\partial \Omega$ by requiring continuity of $u_\delta$ and
$\frac{1}{\e_\delta} \frac{\partial u_\delta}{\partial \nu}. $ Our overall goal is to understand the
dependence of the solution $u_\delta$ on the parameter $\delta$ near $\delta = 0$.

\subsection{The nonresonance condition and three auxiliary functions} \label{s.asspt}

Our analysis requires that the dopant be non-resonant, in other words that $k^2$ not be a Dirichlet eigenvalue of the
Laplacian in $D$. Equivalently:
\begin{align}\label{nonres-cond}
&\mbox{\it We assume that the Helmholtz operator $-\Delta - k^2$} \\[-5pt]
&\mbox{\it has trivial kernel in $H^1_0(D)$.} \nonumber
\end{align}
By elliptic theory, this condition holds for all but a countable collection of $k^2 \in (0,\infty).$ (For some
discussion about what happens when this condition fails, see Section \ref{s.exam}.)

The solutions of the following three auxiliary problems will play key roles in our analysis:
\medskip

\noindent
\textit{Problem 1:} Let $s: \RR^2 \setminus \overline{\Omega} \to \RR$ denote the {unique solution} to
\begin{equation} \label{e.PDEs} \left\{
    \begin{aligned}
        -\Delta s &= k^2 s + f \quad  &\mbox{ in } \RR^2 \setminus \overline{\Omega},\\
        s &= 0 \quad  &\mbox{ along } \partial \Omega ,  \\
        s &\mbox{ satisfies the radiation condition (\ref{e.bcinfty}) at infinity.}
    \end{aligned} \right.
\end{equation}

\noindent
\textit{Problem 2:} Let $\psi_e$ denote the {unique solution} to
\begin{equation} \label{e.PDEpsiext} \left\{
    \begin{aligned}
        -\Delta \psi_e &= k^2 \psi_e  \quad &\mbox{ in } \RR^2 \setminus \overline{\Omega}, \\
        \psi_e &= 1 \quad &\mbox{ along } \partial \Omega, \\
        \psi_e &\mbox{ satisfies the radiation condition (\ref{e.bcinfty}) at infinity. }
    \end{aligned} \right.
\end{equation}

\noindent
\textit{Problem 3:} Let $\psi_d$ denote the unique solution to
\begin{equation} \label{e.PDEpsidop} \left\{
     \begin{aligned}
        -\Delta \psi_d &= k^2 \psi_d  \quad &\mbox{ in } {D}, \\
        \psi_d &= 1 \quad &\mbox{ along } \partial D .
    \end{aligned} \right.
\end{equation}

\noindent
The existence and uniqueness of solutions to Problems 1 and 2 is well-known, see for example \cite[Theorem 9.11]{McL}.
The existence and uniqueness of $\psi_d \in H^1(D)$ follows by Fredholm theory from our nonresonance hypothesis \eqref{nonres-cond}.

Our analysis will require dividing by the complex number
\begin{equation}
    \label{e.c*fin}
  \beta :=  k^2 |\Omega \setminus \overline{D}| + \int_{\partial \Omega} \frac{\partial \psi_e}{\partial \nu_\Omega} \,d \sh^1 - \int_{\partial D} \frac{\partial \psi_d}{\partial \nu_D} \,d \sh^1 .
\end{equation}
(Our sign conventions for the normal derivatives are shown in Fig. \ref{fig:setup}: $\nu_\Omega$ denotes
the outward pointing unit normal to $\partial \Omega,$ and $\nu_D$ denotes the outward pointing unit normal to $\partial D.$) It is therefore important to know that this never vanishes.

\begin{lemma} \label{beta-nonzero} The complex number $\beta$ defined by \eqref{e.c*fin} is always nonzero.
\end{lemma}
\begin{proof}
A result sometimes known as Rellich's lemma says that if $u$ solves $\Delta u + k^2 u = 0$ in $\RR^2 \setminus \overline{\Omega}$
and satisfies the radiation condition \eqref{e.bcinfty} at infinity, then
\begin{multline*}
-2 \, \mbox{Im} \left(k \int_{\partial \Omega} u \frac{\partial \overline{u}}{\partial \nu_\Omega} \, d \sh^1 \right) = \\
2 \, \mbox{Im}(k) \int_{\RR^2 \setminus \overline{\Omega}} ( |k|^2 |u|^2 + |\nabla u|^2 ) \, dx +
\lim_{R \rightarrow \infty} \int_{|x|=R} ( |\partial u/\partial r|^2 + k^2 |u|^2 ) \, d \sh^1
\end{multline*}
(see e.g. equation (3.10) of \cite{CK1}); moreover the second term on the right
can be written in terms of the
far-field pattern of $u$, and it never vanishes (unless $u=0$). Specializing to $u=\psi_e$, we have
\begin{multline} \label{rellich-psi-e}
-2 \, \mbox{Im} \left(k \int_{\partial \Omega} \frac{\partial \overline{\psi}_e}{\partial \nu_\Omega} \, d \sh^1 \right) = \\
2 \, \mbox{Im}(k) \int_{\RR^2 \setminus \overline{\Omega}} (|k|^2 |\psi_e|^2 + |\nabla \psi_e|^2) \, dx +
\lim_{R \rightarrow \infty} \int_{|x|=R} ( |\partial \psi_e/\partial r|^2 + k^2 |\psi_e|^2 ) \, d \sh^1 .
\end{multline}
The function $\psi_d$ satisfies a similar relation: multiplying its PDE by $\overline{\psi}_d$ and integrating by parts gives
$$
\int_{\partial D} \frac{\partial \psi_d}{\partial \nu_D} \, d \sh^1 =
\int_D ( |\nabla \psi_d|^2 - k^2 |\psi_d|^2 ) \, dx .
$$
Taking the complex conjugate then multiplying by $k$ and
taking the imaginary part gives
\begin{equation} \label{rellich-psi-d}
\mbox{Im} \left( k \int_{\partial D} \frac{\partial \overline{\psi}_d}{\partial \nu_D} \, d \sh^1 \right) =
\mbox{Im}(k) \int_D ( |\nabla \psi_d|^2 + |k|^2 |\psi_d|^2 ) \, dx .
\end{equation}

We turn now to the assertion of the lemma. If $k$ is real then $\psi_d$ is real, and it follows from \eqref{rellich-psi-e}
that the imaginary part of $\int_{\partial \Omega} \partial \psi_e / \partial \nu_\Omega \, d \sh^1$ is nonzero. Since the other
terms in \eqref{e.c*fin} are real, $\beta \neq 0$. If $k$ is not real then the imaginary part of $k$ is positive (since
$0 \leqslant \mathrm{arg} \, k < \pi$). From \eqref{e.c*fin} we have
$$
k \overline{\beta} =  \overline{k} \, |k|^2 |\Omega \setminus \overline{D}| +
k \int_{\partial \Omega} \frac{\partial \overline{\psi}_e}{\partial \nu_\Omega} \,d \sh^1 -
k \int_{\partial D} \frac{\partial \overline{\psi}_d}{\partial \nu_D} \, d \sh^1 .
$$
Equations \eqref{rellich-psi-e} and \eqref{rellich-psi-d} show that the imaginary part of the right hand side is
strictly negative, from which it follows that $\beta \neq 0$.
\end{proof}

\subsection{Notation, conventions, and some function spaces}
Unless otherwise specified, all spaces of functions in the sequel consist of complex-valued functions.
All PDEs are to be interpreted in the weak sense, by testing with an $H^1$ function.

We will frequently need to consider the restriction to $\partial \Omega$ of a function in $H^1(\Omega)$, so let us
recall the facts we'll be using.
As the domain $\Omega$ is bounded with Lipschitz boundary, each function $f \in H^1(\Omega)$ admits a trace $\mathrm{tr}(f) \in L^2(\partial \Omega).$ This trace map is a bounded linear functional on $H^1(\Omega)$, whose kernel is the space $H^1_0(\Omega),$ and range is the fractional Sobolev space $H^{1/2}(\partial \Omega)$ (see \cite[Theorem 3.38]{McL}). The space $H^{1/2}(\partial \Omega)$ is a Hilbert space with respect to the norm defined variationally via: for any $g \in H^{1/2}(\partial \Omega),$
\begin{align}  \label{e.Hhalfdef}
    &\|g\|_{H^{1/2}(\partial \Omega)} :=
    \inf\left\{ \|\tilde{G}\|_{H^1(\Omega)} : G \in H^1(\Omega) \mbox{ and } \mathrm{tr}\, G = g\right\}.
\end{align}
For a proof of the assertion that this definition is equivalent to other definitions (say via integral kernels or by locally flattening the boundary, using a partition of unity, and appealing to a Fourier based definition within each flattened patch) of the fractional Sobolev space $H^{1/2}(\partial \Omega)$ within the setting of Lipschitz domains, we refer the reader to \cite[Section 5.1]{Kir}.
Finally, we define $H^{-1/2}(\partial \Omega)$ to be the dual space of $H^{1/2}(\partial \Omega).$ It is equivalently characterized as the space of traces of normal components $V \cdot \nu_\Omega$ of measurable vector fields $V$ on $\Omega$ such that $V \in (L^2(\Omega))^2,$ and the divergence of $V$ is square integrable, i.e., $\mathrm{div}\, V \in L^2(\Omega) $ (see \cite[Chapter 20]{TartarSobolev} for a proof).

An application of this latter equivalence that we will use repeatedly, is an estimate of the normal derivative of $s$ defined via \eqref{e.PDEs}. We let $\Psi \in C^1_c(\RR^2 \setminus \overline{\Omega})$ denote a cut off function that satisfies $\Psi \equiv 1$ in a neighborhood of $\Omega.$ Considering the vector field $V := \Psi  \nabla s ,$ from \eqref{e.PDEs} it is clear that $V \in (L^2(\RR^2\setminus \overline{\Omega}))^2,$ with $\mathrm{div} \, V \in L^2(\RR^2 \setminus \overline{\Omega}),$ has compact support, and finally, satisfies $V \cdot \nu = \frac{\partial s}{\partial \nu}.$ Therefore, by the above characterization of $H^{-1/2}(\partial \Omega),$ it follows that
\begin{align} \label{e.sbound}
\left\| \frac{\partial s}{\partial \nu_\Omega} \right\|_{H^{-1/2}(\partial \Omega)} &=
\| V \cdot \nu_\Omega\|_{H^{-1/2}(\partial \Omega)}\\  &\leqslant
C(\|V\|_{L^2} + \|\mathrm{div}\, V\|_{L^2}) \leqslant
C \|f\|_{L^2(\RR^2 \setminus \overline{\Omega})}, \nonumber
\end{align}
using the estimate
\begin{align*}
    \int_{\mathrm{spt}\, \Psi} |s|^2 + |\nabla s|^2 \,dx \leqslant C\int_{\RR^2 \setminus \overline{\Omega}} |f|^2 \,dx ,
\end{align*}
which is proved in showing the existence and uniqueness of $s$ (see for example \cite{McL}).
 \\
\subsection{A summary of our main results} \label{ss.main}
Our main theorem is
 \begin{theorem} \label{t.mt}
Suppose the dopant is non-resonant, in other words that \eqref{nonres-cond} holds. Then
there exists $\Lambda > 0$, depending only on $k$ and the sets $\Omega$ and $D$
(and \textbf{independent} of the source $f$ and the value of $\delta$) such that
if $|\delta | < \frac{1}{\Lambda},$ then the map $\delta \mapsto u_\delta \in H^1_{loc}(\RR^2;\mathbb{C})$
is complex analytic. To be precise, there exist $\{v_j\}_{j=0}^\infty \subset H^1_{loc}(\RR^2),$
independent of $\delta$, such that
 \begin{align}  \label{e.ud.exp}
     u_\delta = \sum_{j=0}^\infty \delta^j v_j \quad \quad \mbox{ in } H^1_{loc}(\RR^2),
 \end{align}
with $v_j$ satisfying: for each open set $U \subset \RR^2,$ there exists a constant $K_U > 0$ (independent of $\delta$), such that
 \begin{align*}
     \|v_j\|_{H^1(U)} \leqslant K_U \Lambda^j\left\| f \right\|_{L^2}.
 \end{align*}
 \end{theorem}
As will be clear from the proof, we obtain a complete characterization of all the ``correctors'' $v_j$ that
appear in the statement of Theorem \ref{t.mt}. But it is the leading-order correctors that are the most important, and the
easiest to describe. Therefore we devote the remainder of this section to a self-contained description of the leading-order
correctors, and statements of some corollaries that they permit us to draw.

The zeroth order term in \eqref{e.ud.exp} is of course the easiest to describe: it is
\begin{align*}
    v_0(x) := \left\{
    \begin{array}{cc}
        c^* \psi_e + s & x \in \RR^2 \setminus \overline{\Omega}  \\
        c^* &  x \in \Omega \setminus \overline{D}\\
        c^* \psi_d & x \in D,
    \end{array}
    \right.
\end{align*}
where
\begin{equation} \label{e.cstardef}
    c^* := -\frac{1}{\beta}\int_{\partial \Omega} \frac{\partial s}{\partial \nu_\Omega} \,d \sh^1,
\end{equation}
which is finite since $\beta \neq 0$ by {Lemma \ref{beta-nonzero}}.

\begin{remark}
\label{r.mueff} As noted in the introduction, $v_0$ is constant in the ENZ region. In the physics literature, the term ``effective permeability'' $\mu_{\mathrm{eff}}$ is used for the choice of $\mu$ that would have given the same value to $v_0$ in the ENZ region without any dopant. In our setting, the value of this choice is
\begin{align} \label{e.mueff.1}
    \mu_{\mathrm{eff}} := \frac{|\Omega \setminus \overline{D}| + \int_D \psi_d \,dx }{|\Omega|}\mu .
\end{align}
Indeed, the preceding description of $\mu_{\mathrm{eff}}$ reduces to the value such that
\begin{align} \label{e.mueff.2}
  \omega^2  \mu_{\mathrm{eff}}|\Omega| + \int_{\partial \Omega} \frac{\partial \psi_e}{\partial \nu_\Omega}\,d \sh^1 = \omega^2 \mu |\Omega \setminus \overline{D}|  + \int_{\partial \Omega} \frac{\partial \psi_e}{\partial \nu_\Omega}\,d \sh^1 -  \int_{\partial D} \frac{\partial \psi_d}{\partial \nu_D}\,d \sh^1.
\end{align}
A bit of manipulation using Green's theorem, the PDE solved by $\psi_d$, our convention that $k^2 = \omega^2 \mu$, and
$$
    -\int_{\partial D} \frac{\partial \psi_d}{\partial \nu_D}\,d \sh^1 = k^2 \int_D \psi_d \,dx,
$$
reveals that \eqref{e.mueff.1} and \eqref{e.mueff.2} are equivalent.
\end{remark}

To identify the leading order electric field $\frac{1}{i \omega \e(x)} (-\partial_{x_2} u_\delta, \partial_{x_1} u_\delta , 0)$ we must discuss the $j = 1$ term in \eqref{e.ud.exp}. To do so, we introduce new correctors.
We let $\phi_0 \in H^1(\Omega \setminus \overline{D})$ denote the unique weak solution to the PDE
\begin{equation} \label{e.phizerodef} \left\{
    \begin{aligned}
    -\Delta \phi_0 &= k^2 c^* \quad \quad \mbox{ in } \Omega \setminus \overline{D} \\
     \frac{\partial \phi_0}{\partial \nu_\Omega} &= c^* \frac{\partial \psi_e}{\partial \nu_\Omega} + \frac{\partial s}{\partial \nu_\Omega} \quad \quad \mbox{ on } \partial \Omega \\
    \frac{\partial \phi_0}{\partial \nu_D} &= c^* \frac{\partial \psi_d}{\partial \nu_D} \quad \quad \mbox{ on } \partial D,\\
     \fint_{\Omega \setminus \overline{D}} & \phi_0 \,dx = 0;
    \end{aligned} \right.
\end{equation}
and we define the exterior corrector $\lambda_0 \in H^1(\RR^2 \setminus \overline{\Omega}),$
\begin{equation} \label{e.lam0defintro} \left\{
    \begin{aligned}
    -\Delta \lambda_0 &= k^2 \lambda_0 \quad \quad \mbox{ in } \RR^2 \setminus \overline{\Omega}, \\
    \lambda_0 &= \phi_0 \quad \quad \mbox{ on } \partial \Omega,\\
    \lambda_0 &\mbox{ satisfies the radiation condition (\ref{e.bcinfty}) at infinity},
    \end{aligned} \right.
\end{equation}
and the dopant corrector $\chi_0 \in H^1(D),$ given by
\begin{equation}
    \label{e.chi0defintro} \left\{
  \begin{aligned}
  -\Delta \chi_0 &= k^2 \chi_0 \quad \quad \mbox{ in } D, \\
  \chi_0 &= \phi_0 \quad \quad \mbox{ on } \partial D.
  \end{aligned}
    \right.
\end{equation}
We emphasize that the corrector $\phi_0$ in the ENZ region $\Omega \setminus \overline{D}$ solves a \textit{Poisson equation with constant right hand side}, not a Helmholtz equation. The choice of $c^*$ by \eqref{e.cstardef} exactly ensures that the problem \eqref{e.phizerodef} is well-posed; the problems defining $(\lambda_0,\chi_0)$ are also well-posed, since we have assumed that the dopant isn't resonant. Define
\begin{equation} \label{e.e0defintro}
    e_0 := -\frac{1}{\beta}\left( \int_{\partial \Omega} \frac{\partial \lambda_0}{\partial \nu_\Omega} \,d \sh^1 - \int_{\partial D} \frac{\partial \chi_0}{\partial \nu_D}\,d \sh^1\right),
\end{equation}
and set $c_\delta := c^* + \delta e_0.$ The following result captures the essential information available from the first two terms in \eqref{e.ud.exp}.
\begin{corollary} \label{c.leadord1}
Let $|\delta| < \frac{1}{\Lambda},$ where $\Lambda$ is as in Theorem \ref{t.mt}, and let $v_\delta \in H^1_{loc}(\RR^2)$ be defined via
\begin{equation}
    v_\delta(x) := \left\{
    \begin{array}{ccc}
        &c_\delta \psi_e (x)+ s (x) + \delta \lambda_0(x) & x \in \RR^2 \setminus \overline{\Omega}  \\
        &c_\delta + \delta \phi_0 & x \in \Omega \setminus \overline{D}\\
        &c_\delta \psi_d + \delta \chi_0 & x \in D.
    \end{array}
    \right.
\end{equation}
Then, for any open set $U \subset \RR^2$ with compact closure, there exists a constant $C_U > 0$ independent of $\delta$ such that
\begin{align} \label{e.quadest}
    \|u_\delta - v_\delta\|_{H^1(U)} \leqslant C_U \delta^2.
\end{align}
\end{corollary}
Let us note that we obtain local in $H^1$ estimates only because of the unboundedness of the domain. In particular, within the ENZ region $\Omega \setminus \overline{D},$ the differences above are global, in the following sense: the estimate \eqref{e.quadest} says that $u_\delta$ is close to being a constant, where closeness is measured in terms of $H^1(\Omega \setminus \overline{D})$.
\begin{remark}
\label{r.electricfield}
A consequence of Corollary \ref{c.leadord1} is that at leading order, the electric field distribution in the ENZ region is given by
\begin{align*}
    \vec{E} := \frac{1}{i \omega} (-\partial_{x_2} \phi_0, \partial_{x_1} \phi_0, 0).
\end{align*}
Indeed, this follows, since within the TM setting of our paper, the magnetic field is $\vec{H} = (0,0,u_\delta),$ and the electric field is $\vec{E} = \frac{1}{i \omega \e_\delta(x)}(-\partial_{x_2} u_\delta, \partial_{x_1}u_\delta, 0).$ Similarly, from Corollary \ref{c.leadord1}, it is also clear that one can read off information about the electric field distribution in the other regions.
\end{remark}

Our next corollary generalizes the so-called ``Ideal Fluid Analogy'' introduced in \cite{liberal2020near}. In that paper, the authors consider the setting in which both the dielectric permittivity $\e_\delta$ and the magnetic permeability ($k^2 = \omega^2 \mu$ in the notation of the introduction), both vanish in the ENZ region. As explained in the introduction, our next corollary clarifies and complements the result in two ways: first, of course, it describes corrections to the leading order theory recalled in the introduction  due to nonzero $\delta$ (e.g., via { the presence of losses}), and second, it is not restricted to the case of vanishing magnetic permeability. In order to state this corollary, we recall the notion of the Poynting vector, as the authors in \cite{liberal2020near} formulate their ideal fluid analogy in terms of the Poynting vector. It is defined as the function $S_\delta := \frac{1}{2} \vec E_\delta \times \overline{\vec H_\delta},$ with $\overline{z}$ denoting the complex conjugate of a complex number $z \in \mathbb{C}$ . We have the following characterization of the limit  of $S_\delta$ as $\delta \to 0:$
\begin{corollary} \label{c.idealfluid}
In the ENZ set $\Omega \setminus \overline{D},$ the Poynting vector field $S_\delta$ converges in $L^2(\Omega \setminus \overline{D})$ strongly to ${S} =  \frac{\overline{c^*}}{2i \omega} \nabla \phi_0$, which satisfies the PDE system
\begin{equation} \left\{\begin{aligned}
    \nabla \cdot {S} &=- \frac{1}{2i \omega} k^2 |c^*|^2 = \frac{i\omega \mu}{2 }|c^*|^2, \quad \quad \mbox{ in } \Omega \setminus \overline{D}, \\
    \nabla \times {S} &= 0, \quad \quad \mbox{ in } \Omega \setminus \overline{D}, \\
    \nu_\Omega \cdot {S} &=  \frac{1}{2i \omega} \overline{c^*} \left( c^* \frac{\partial \psi_e}{\partial \nu_\Omega} + \frac{\partial s}{\partial \nu_\Omega} \right), \quad \quad \mbox{ on } \partial \Omega, \\
     \nu_D \cdot {S} &= \frac{1}{2i\omega}|c^*|^2 \frac{\partial \psi_d}{\partial \nu_D} \quad \quad \mbox{ on } \partial D.
        \end{aligned} \right. \end{equation}
\end{corollary}
The ``ideal fluid analogy'' explored in \cite{liberal2020near} rests on the observation that if $\mu=0$ in the ENZ region then $\nabla \cdot S = 0$ and $\nabla \times S = 0$ there (so the real and imaginary parts of $S$ are the gradients of harmonic functions). Corollary \ref{c.idealfluid} provides a simple extension to our setting, which applies for any value of $\mu.$ In particular, the real and imaginary parts of $S$ in the ENZ region are the gradients of functions satisfying Poisson-type equations of the form $\Delta W = \mbox{constant},$ where the right-hand-side is the real or imaginary part of $\tfrac12 i \omega \mu |c^*|^2. $

\subsection{Main ingredients of the proof} \label{ss.mainidea}
The main idea in the proof of Theorem \ref{t.mt} and its corollaries is to pursue a perturbative ansatz of the form
\begin{align} \label{e.expansion}
    u_\delta = \sum_{j=0}^\infty \delta^j v_j,
\end{align}
and plug it into the PDE \eqref{e.PDE}. When $|\delta| \ll 1$, we expect that $u_\delta$ is essentially constant in the ENZ region. To this end,  by linearity of the PDE \eqref{e.PDE}, one seeks
\begin{align*}
    v_0 := \left\{
    \begin{array}{cc}
     c \psi_e + s    & \mbox{ in } \RR^2 \setminus \overline{\Omega}\\
    c     &  \mbox{ in } \Omega \setminus \overline{D}\\
    c \psi_d & \mbox{ in } D,
    \end{array}
    \right.
\end{align*}
for an unknown constant $c,$ where the functions $s,\psi_e,$ and $\psi_d$ are as defined via \eqref{e.PDEs}, \eqref{e.PDEpsiext} and \eqref{e.PDEpsidop}. Observing that continuity of $v_0$ across $\partial \Omega$ and $\partial D$ is enforced by construction, it remains to determine the unknown constant $c.$ The value of this constant is determined at the next order. Indeed, plugging in \eqref{e.expansion} into \eqref{e.PDE}, and matching the terms at $O(\delta)$ and enforcing continuity of $\e_\delta^{-1}\frac{\partial u_\delta}{\partial \nu_\Omega}$ (respectively $\e_\delta^{-1} \frac{\partial u_\delta}{\partial \nu_D})$ at the interface $\partial \Omega$ (respectively $\partial D$),  one arrives at the system \eqref{e.phizerodef} for $\phi_0 := v_1|_{\Omega \setminus \overline{D}}$ in the ENZ region, and the systems \eqref{e.lam0defintro} and \eqref{e.chi0defintro} respectively for the restrictions $\lambda_0$ and $\chi_0$ of the corrector $v_1$ in the regions $\RR^2 \setminus \overline{\Omega}$ and $D$ respectively.
{We note}
that the PDE \eqref{e.phizerodef} is a Poisson equation, and the value $c = c^*$ is exactly the unique choice of the constant $c$ that makes this problem consistent. The choice of the boundary conditions for $\phi_0$ ensure continuity of $\e_\delta^{-1} \frac{\partial v_0}{\partial \nu}$ across both the interfaces, and those for $\lambda_0$ and $\chi_0$ yield continuity of $v_1$ across the interfaces. The subsequent order correctors are then defined inductively, where at each stage, the constant $c^*$ is corrected by suitable constants $\delta^j e_j,$ that render the Poisson problems in the ENZ region consistent, providing the Dirichlet data for the exterior and dopant Helmholtz problems.

The bulk of the proof lies in setting up the induction procedure above, and arguing that the system for the infinite system of correctors does indeed close, with bounds sufficient to ensure analyticity. The key point, as one might naturally expect in the context of divergence form PDEs with piecewise constant coefficients, is that the regions $\Omega \setminus \overline{D}$, $D$ and $\RR^2 \setminus \overline{\Omega}$ interact with each other through their Dirichlet-to-Neumann maps.

\section{An Operator Theoretic Set-up} \label{s.opt}
We introduce various operators that will {be used in the proof} of our main result.

\subsection{Dirichlet-to-Neumann operators outside the ENZ region}
Let $\mathcal{D}^e_k$ and $\mathcal{D}^d_k$ denote the Dirichlet-to-Neumann operators associated with the Helmholtz operator $-\Delta - k^2$ in the exterior and in the dopant regions respectively. To be precise, given $g_e \in H^{1/2}(\partial \Omega)$ (resp. $g_d \in H^{1/2}(D)$), define these operators via
\begin{equation*} \left\{
    \begin{aligned}
     -\Delta u_e &= k^2 u_e \quad \mbox{ in } \RR^2 \setminus \overline{\Omega}, \\
     u_e &= g_e  \quad \mbox{ on } \partial \Omega, \\
     u_e &\mbox{ satisfies the radiation condition (\ref{e.bcinfty}) at infinity }\\
      \mathcal{D}^e_k (g_e) &:= \frac{\partial u_e}{\partial \nu_\Omega} \in H^{-1/2}(\partial \Omega),
    \end{aligned} \right.
\end{equation*}
and
\begin{equation*} \left\{
    \begin{aligned}
     -\Delta u_d &= k^2 u_d \quad \mbox{ in } D, \\
     u_d &= g_d  \quad \mbox{ on } \partial D, \\
       \mathcal{D}^d_k (g_d) &:= \frac{\partial u_d}{\partial \nu_D} \in H^{-1/2}(\partial D),
    \end{aligned} \right.
\end{equation*}
respectively. A basic fact that we will use from \cite{CK1,CK2,McL} is
\begin{lemma} \label{l.dtn}
 Under the assumptions in Section \ref{s.asspt}, the operators $\mathcal{D}^e_k$ and $\mathcal{D}^d_k$ define bounded linear operators: to be precise, there exist constants $C_k^e, C_k^d$ such that
\begin{align*}
    &\|\mathcal{D}^e_k(g_e)\|_{H^{-1/2}(\partial \Omega )} \leqslant C_k^e \|g_e\|_{H^{1/2}(\partial \Omega)}, \\
    & \|\mathcal{D}^d_k(g_d)\|_{H^{-1/2}(\partial D )} \leqslant C_k^d \|g_d\|_{H^{1/2}(\partial D)},
\end{align*}
\end{lemma}

With this definition at hand, by a convenient abuse of notation, we will write
\begin{align*}
    \int_{\partial \Omega} \mathcal{D}^e_k(g)\,d \sh^1 := \langle \mathcal{D}^e_k(g) , \mathbf{1} \rangle_{H^{-1/2}(\partial \Omega), H^{1/2}(\partial \Omega)},
\end{align*}
where $\mathbf{1}(x) \equiv 1$ for all $x \in \partial \Omega$
and, we use a similar convention at $\partial D.$ By duality,
\begin{equation} \label{e.avgbnd}
    \begin{aligned}
    &\left| \int_{\partial \Omega} \mathcal{D}^e_k(g_e)\,d \sh^1 \right| \leqslant C_k^e \|g_e\|_{H^{1/2}(\partial \Omega)} \|\mathbf{1}\|_{H^{1/2}(\partial \Omega)}, \\
    &\left| \int_{\partial D} \mathcal{D}^d_k(g_d)\,d \sh^1 \right| \leqslant C_k^d \|g_d\|_{H^{1/2}(\partial D)} \|\mathbf{1}\|_{H^{1/2}(\partial D)}.
    \end{aligned}
\end{equation}
A tensorized version of $\mathcal{D}_k^e,\mathcal{D}_k^d$ is easily defined via: given $(g_e,g_d) \in H^{1/2}(\partial \Omega) \times H^{1/2}(\partial D),$ set
\begin{align*}
    \mathcal{D}_k(g_e,g_d) := (\mathcal{D}_k^e(g_e), \mathcal{D}_k^d(g_d)).
\end{align*}
\subsection{An averaging operator} For any $(h_e,h_d) \in H^{-1/2}(\partial \Omega) \times H^{-1/2}(\partial D),$ we define
\begin{align*}
    \mathcal{A}(h_e,h_d) :=  -\frac{1}{\beta} \left( \int_{\partial \Omega} h_e\,d \sh^1 -  \int_{\partial D} h_d\,d \sh^1 \right) \in \mathbb{C},
\end{align*}
recalling that $\beta \neq 0$ by Lemma \ref{beta-nonzero}.
\begin{lemma} \label{l.avg}
The functional $\mathcal{A}: H^{1/2}(\partial \Omega) \times H^{1/2}(\partial D) \to \mathbb{C}$ is a linear map that is bounded in the sense that there exists a constant $C_{\mathcal{A}} > 0$ with
\begin{align}
    |\mathcal{A}(h_e,h_d)| \leqslant C_{\mathcal{A}} (\|h_e\|_{H^{-1/2}(\partial \Omega)} + \|h_d\|_{H^{-1/2}(\partial D)} ).
\end{align}
\end{lemma}
\begin{proof}
The desired estimate follows easily from \eqref{e.avgbnd} and the triangle inequality.
\end{proof}

\subsection{The solution operators}
We introduce solution operators in the various regions. In order to define this operator, we let $\zmh(\Omega \setminus \overline{D})$ denote the set of square integrable functions in $\Omega \setminus \overline{D},$ whose average is zero: precisely,
$$
\zmh(\Omega \setminus \overline{D}) := \left\{f \in L^2(\Omega \setminus \overline{D}): \int_{\Omega \setminus \overline{D}} f = 0\right\},
$$
and
\begin{align*}
    \underline{H}^1(\Omega \setminus \overline{D}) := \left\{f \in H^1(\Omega \setminus \overline{D}): \int_{\Omega \setminus \overline{D}} f = 0\right\}\,.
\end{align*}

We define the solution operator in the ENZ region, denoted by $\mathcal{P}_k: \zmh(\Omega \setminus \overline{D}) \times H^{-1/2} (\partial \Omega) \times H^{-1/2}(\partial D) \to  \underline{H}^1(\Omega \setminus \overline{D})$ by
\begin{align*}
    \mathcal{P}_k(g,h_e,h_d) := \hat\phi ,
\end{align*}
where $\hat\phi = \hat\phi(g,h_e,h_d)$ is the unique $H^1$ weak solution to the boundary value problem
\begin{equation} \label{e.Skopdef} \left\{
    \begin{aligned}
   -\Delta \hat\phi &= k^2 \mathcal{A}(h_e,h_d) + k^2 g, \quad \quad \mbox{ in } \Omega \setminus \overline{D},\\
      \frac{\partial \hat\phi}{\partial \nu_\Omega} &= \mathcal{A}(h_e,h_d) \frac{\partial \psi_e}{\partial \nu} + h_e  \quad \quad \mbox{ on } \partial \Omega, \\
      \frac{\partial \hat\phi}{\partial \nu_D} &= \mathcal{A}(h_e,h_d) \frac{\partial \psi_d}{\partial \nu_D} + h_d, \quad \quad \mbox{ on } \partial D, \\
      \fint_{\Omega \setminus \overline{D}} & \hat\phi = 0.
    \end{aligned} \right.
\end{equation}
The choice of the constant $\mathcal{A}(h_e,h_d)$ is exactly such that this system is consistent.

In order to define correctors in the exterior and in the dopant region, we define the Helmholtz solution operators $\mathcal{H}_k^e : H^{1/2}(\partial\Omega) \to H^1_{loc}(\RR^2 \setminus \overline{\Omega})$, and respectively, $\mathcal{H}_k^d : H^{1/2}(\partial D) \to H^1(D),$ via
\begin{equation} \label{e.Hedef} \left\{
    \begin{aligned}
    -\Delta v &= k^2 v \quad \quad \mbox{ in } \RR^2 \setminus \overline{\Omega}\\
    v &= g_e \in H^{1/2}(\partial \Omega), \\
     v &\mbox{ satisfies the radiation condition (\ref{e.bcinfty}) at infinity}\\
     \mathcal{H}_k^e&(g_e) := v,
    \end{aligned} \right.
\end{equation}
and respectively,
\begin{equation} \label{e.Hddef} \left\{
    \begin{aligned}
    -\Delta v &= k^2 v \quad \quad \mbox{ in } D\\
    v &= g_d \in H^{1/2}(\partial D),\\
     \mathcal{H}_k^d&(g_d) := v,
    \end{aligned} \right.
\end{equation}
That $\mathcal{H}^d_k$ is well-defined is guaranteed by our assumption \eqref{nonres-cond} that the dopant isn't resonant.
As usual, we tensorize this to define $\mathcal{H}_k : H^{1/2}(\partial \Omega)\times H^{1/2}(\partial D) \to H^{1}_{loc}(\RR^2 \setminus \overline{\Omega})\times H^{1}( D)$ via
\begin{align}
    \mathcal{H}_k(g_e,g_d) = (\mathcal{H}_k^e(g_e), \mathcal{H}_k^d(g_d)).
\end{align}
\begin{lemma} \label{l.extbd}
The operator $\mathcal{P}_k : \zmh(\Omega \setminus \overline{D}) \times H^{-1/2} (\partial \Omega) \times H^{-1/2}(\partial D) \to H^1(\Omega \setminus \overline{D}),$ and $\mathcal{H}_k : H^{1/2}(\partial \Omega) \times H^{1/2}(D) \to H^1_{loc}(\RR^2 \setminus \overline{\Omega}) \times H^1(D)$ are both bounded linear operators between the respective Hilbert spaces.
\end{lemma}
\begin{proof}
See \cite[Theorem 8.19]{McL} for the assertion regarding $\mathcal{P}_k$, and \cite[Theorem 9.11 and page 286]{McL} for the statement regarding $\mathcal{H}_k.$
\end{proof}

Denoting the boundary trace operator for the ENZ region, $\Omega \setminus \overline{D}$ by $\mathcal{T},$ it then follows that the composition  $\mathcal{T}\mathcal{P}_k: \zmh(\Omega \setminus \overline{D}) \times H^{-1/2}(\partial \Omega)\times H^{-1/2}(\partial D) \to H^{1/2}(\partial\Omega) \times H^{1/2}(\partial D)$ is a bounded linear operator. The main object of our proof is the iteration map $\mathcal{I}_k: \underline{H}^1(\Omega \setminus \overline{D}) \times H^{-1/2}(\partial \Omega) \times H^{-1/2} (\partial D) \to \underline{H}^1(\Omega \setminus \overline{D}) \times H^{-1/2}(\partial \Omega) \times H^{-1/2} (\partial D)$ defined via
\begin{align} \label{e.Ikdef}
   & (g,h_e,h_d) \in \underline{H}^1(\Omega \setminus \overline{D}) \times H^{-1/2}(\partial \Omega) \times H^{-1/2} (\partial D) \\ \notag &\quad \quad \quad \quad \quad \mapsto \mathcal{I}_k(g,h_e,h_d) := (\mathcal{P}_k (g,h_e,h_d) , \mathcal{D}_k \mathcal{T} \mathcal{P}_k (g,h_e,h_d)).
\end{align}
In particular, $\mathcal{I}_k$ is a bounded linear operator.

Recall the number $c^*$ defined in \eqref{e.cstardef}, and the leading order correctors $(\phi_0, \lambda_0,\chi_0)$ defined in \eqref{e.phizerodef}, \eqref{e.lam0defintro}, and \eqref{e.chi0defintro} respectively. We note that
\begin{align} \label{e.c*A}
    c^* = \mathcal{A} \left( \frac{\partial s}{\partial \nu_\Omega}, 0 \right),
   \end{align}

\begin{equation} \label{e.phi0def}
     \phi_0 := \mathcal{P}_k \left(0,\frac{\partial s}{\partial \nu_\Omega} , 0\right),
\end{equation}
     and that

   \begin{equation} \label{e.lam0chi0def}
     \begin{aligned}
     (\lambda_0,\chi_0) = \mathcal{H}_k\mathcal{T}  \mathcal{P}_k \left( 0,\frac{\partial s}{\partial \nu_\Omega}, 0 \right).
     \end{aligned}
 \end{equation}
 Finally, we note
 \begin{equation} \label{e.zeronormalder}
     \begin{aligned}
     \left(\frac{\partial \lambda_0}{\partial \nu_\Omega}, \frac{\partial \chi_0}{\partial \nu_D} \right) = \mathcal{D}_k \mathcal{T}\mathcal{P}_k \left( 0,\frac{\partial s}{\partial \nu_\Omega}, 0 \right),
     \end{aligned}
 \end{equation}
 so that the number $e_0$ defined in \eqref{e.e0defintro} satisfies
 \begin{equation} \label{e.e0opt}
     \begin{aligned}
      e_0 = \mathcal{A} \left( \mathcal{D}_k \mathcal{T}\mathcal{P}_k\left( 0,\frac{\partial s}{\partial \nu_\Omega}, 0\right)\right).
     \end{aligned}
 \end{equation}

 \section{Proofs of the Main Theorem and its Corollaries} \label{s.proofs}
 \begin{proof}[Proof of Theorem \ref{t.mt}]
 We proceed in several steps, deriving recursive definitions of higher order correctors to all orders. Before setting up the induction step, it is convenient to define
 \begin{align}
     \label{e.basecase}
     \phi_{-1}\equiv  0, \, \lambda_{-1} := s,\,  \chi_{-1} \equiv 0, \, \mbox{ and } e_{-1} := c^* = \mathcal{A}\left(\frac{\partial s}{\partial \nu_\Omega},0 \right),
 \end{align}
 and recall the definitions for $(\phi_0, \lambda_0, \chi_0, e_0)$ from \eqref{e.phi0def}--\eqref{e.e0opt}.\\

\textbf{1.} In this step, we define a hierarchy of correctors. For each $j = 1,2,\ldots,$ we inductively define,

\begin{equation}\label{e.highcorsys}
    \begin{aligned}
    \phi_j &:= \mathcal{P}_k \left( \phi_{j-1}, \frac{\partial \lambda_{j-1}}{\partial \nu_\Omega}, \frac{\partial \chi_{j-1}}{\partial \nu_D} \right), \\
    (\lambda_j, \chi_j) &= \mathcal{H}_k\mathcal{T}\mathcal{P}_k \left( \phi_{j-1},\frac{\partial \lambda_{j-1}}{\partial \nu_\Omega}, \frac{\partial \chi_{j-1}}{\partial \nu_D} \right),
    \end{aligned}
\end{equation}
and so, observing that
\begin{equation}
    \label{e.normderj}
    \left( \frac{\partial \lambda_j}{\partial \nu_\Omega}, \frac{\partial \chi_j}{\partial \nu_D} \right) = \mathcal{D}_k \mathcal{T}\mathcal{P}_k \left( \phi_{j-1}, \frac{\partial \lambda_{j-1}}{\partial \nu_\Omega}, \frac{\partial \chi_{j-1}}{\partial \nu_D} \right),
\end{equation}
we set \begin{align} \label{e.ejdef}
    e_j := -\frac{1}{\beta} \left[ \int_{\partial \Omega} \frac{\partial \lambda_{j}}{\partial \nu_\Omega} \,d \sh^1 - \int_{\partial D} \frac{\partial \chi_{j}}{\partial \nu_D}\,d \sh^1 \right] = \mathcal{A} \left( \frac{\partial \lambda_{j}}{\partial \nu_\Omega}, \frac{\partial \chi_{j}}{\partial \nu_D}\right), \quad \quad j \in \mathbb{N}   ,
\end{align}
where we recall that, by \eqref{e.c*fin}, that $\beta  \neq 0.$
\noindent
Therefore, by definition of $\mathcal{I}_k$, and the base case \eqref{e.basecase}, it follows that
\begin{align} \label{e.Ik}
   \left(\phi_j, \frac{\partial \lambda_j}{\partial \nu_\Omega},\frac{\partial \chi_j}{\partial \nu_D}\right) = \mathcal{I}_k^{j+1}\left(0, \frac{\partial s}{\partial \nu_\Omega} , 0\right), \quad j \in \mathbb{N}.
\end{align}

\noindent
 By induction, \eqref{e.zeronormalder}, and \eqref{e.normderj}, and \eqref{e.Ik},
\begin{equation*}
\begin{aligned}
 & e_j = \mathcal{A} \left( \frac{\partial \lambda_j}{\partial \nu_\Omega}, \frac{\partial \chi_j}{\partial \nu_D}\right) \\
 &\quad = \mathcal{A} \left( \mathcal{D}_k \mathcal{T}\mathcal{P}_k\left(\phi_{j-1},\frac{\partial \lambda_{j-1}}{\partial \nu_\Omega}, \frac{\partial \chi_{j-1}}{\partial \nu_D}\right) \right)  = \mathcal{A} \left( \mathcal{D}_k \mathcal{T}\mathcal{P}_k \mathcal{I}_k^j \left(0, \frac{\partial s}{\partial \nu_\Omega}, 0 \right) \right) .
\end{aligned}
\end{equation*}
Similarly, combining \eqref{e.highcorsys} and \eqref{e.Ik},
\begin{equation} \label{e.highcorrrec}
    \begin{aligned}
    \phi_j  &= \mathcal{P}_k \mathcal{I}_k^j \left( 0,\frac{\partial s}{\partial \nu_\Omega}, 0\right),\\
    (\lambda_j, \chi_j) &= \mathcal{H}_k \mathcal{T}\mathcal{P}_k \mathcal{I}_k^j \left( 0, \frac{\partial s}{\partial \nu_\Omega}, 0\right).
    \end{aligned}
\end{equation}
As always, in the foregoing equations, an exponent on an operator means repeated composition.\\

\textbf{2.} In this step, we obtain bounds on the correctors constructed in Step 1. By Lemmas \ref{l.dtn} and \eqref{l.extbd},  $\mathcal{I}_k$ defined in \eqref{e.Ikdef}  is a bounded linear operator; we let $\Lambda$ denote its operator norm.  By Lemmas \ref{l.avg} and  \eqref{e.sbound} we find  \footnote{Here, for two sequences $\{a_j\}_j, \{b_j\}_j \subset \RR,$ we use the notation $a_j \lesssim b_j$ to mean $a_j \leqslant C b_j, j \in \mathbb{N},$ for a universal constant $C$ that is independent of $j.$}
\begin{align}  \label{e.ejest}
    |e_j| \lesssim \Lambda^j \left\| \frac{\partial s}{\partial \nu_\Omega}\right\|_{H^{-1/2}(\partial \Omega)} \lesssim \Lambda^{j} \|f\|_{L^2}.
\end{align}
Similarly, from \eqref{e.highcorrrec}, and Lemmas \ref{l.dtn}, \ref{l.extbd}, and \eqref{e.sbound}, we estimate
\begin{equation}\begin{aligned} \label{e.jcorrest}
&\|\phi_j\|_{H^1(\Omega \setminus \overline{D})} \lesssim  \Lambda^{j} \|f\|_{L^2}, \quad \|\chi_j\|_{H^1(D) }\lesssim  \Lambda^{j}\|f\|_{L^2}, \quad \mbox{and}\\
&\|\lambda_j\|_{H^1(U)} \lesssim  \Lambda^{j} \|f\|_{L^2}
\end{aligned}\end{equation}
for all open sets $U \subset \RR^2 \setminus \overline{\Omega},$ with compact closure.

Let $\delta \in \mathbb{C}$ be such that $|\delta| < \frac{1}{\Lambda}.$ For such $\delta,$ the operator $(I - \delta \mathcal{I}_k)$ is invertible, and its inverse defines a bounded linear operator on the space $\underline{H}^1(\Omega \setminus \overline{D}) \times H^{-1/2}(\partial \Omega) \times H^{-1/2}(\partial D).$ It is a standard fact that the associated Neumann series converges absolutely on compact subsets of $\{|\delta| \leqslant \frac{1}{\Lambda}\}$ . Using the linearity of the operators $\mathcal{A}, \mathcal{D}_k,\mathcal{T}, \mathcal{P}_k$ and the definition of the Neumann series, and \eqref{e.c*A}, we find
\begin{equation} \label{e.cdeldef} \begin{aligned}
    &\hat{c}_\delta := c^* + \sum_{j=0}^\infty \delta^{j+1} e_j = c^* + \sum_{j=0}^\infty \delta^{j+1} \mathcal{A} \left( \mathcal{D}_k \mathcal{T}\mathcal{P}_k \mathcal{I}_k^j\left( 0,\frac{\partial s}{\partial \nu_\Omega}, 0\right)  \right) \\
    & \quad \quad = c^* + \delta\mathcal{A} \left( \mathcal{D}_k \mathcal{T}\mathcal{P}_k \sum_{j=0}^\infty \delta^{j}\mathcal{I}_k^j \left(0, \frac{\partial s}{\partial \nu_\Omega}, 0\right)  \right) \\
    & \quad \quad =c^* + \delta \mathcal{A} \mathcal{D}_k \mathcal{T}\mathcal{P}_k (I - \delta \mathcal{I}_k)^{-1} \left(0, \frac{\partial s}{\partial \nu_\Omega}, 0\right) ,
\end{aligned}\end{equation}
satisfies $|\hat{c}_\delta| < \infty$ for all $|\delta| < \frac{1}{\Lambda}$. Next, we define, and compute using \eqref{e.phi0def} and \eqref{e.highcorrrec} that
\begin{equation} \label{e.phideldef}
\begin{aligned}
  &\hat{\phi}_\delta :=  \sum_{j=0}^\infty \delta^j \phi_j = \sum_{j=0}^\infty \delta^j \mathcal{P}_k \mathcal{I}_k^j \left(0,\frac{\partial s}{\partial \nu_\Omega}, 0 \right) \\
  &\quad = \mathcal{P}_k \left( (I - \delta \mathcal{I}_k)^{-1} \left(0,\frac{\partial s}{\partial \nu_\Omega} ,0 \right)\right)
\end{aligned}
\end{equation}
where, the convergence is understood in the sense of the Hilbert space $H^1(\Omega \setminus \overline{D})$.
Finally,  we define, and compute using \eqref{e.lam0chi0def} and \eqref{e.highcorrrec} that
\begin{equation}
    \begin{aligned}
    &(\hat{\lambda}_\delta,\hat{\chi}_\delta) :=  \sum_{j=0}^\infty \delta^j (\lambda_j,\chi_j) \\
    &\quad \quad = \sum_{j=0}^\infty \delta^j  \mathcal{H}_k \mathcal{T}\mathcal{P}_k \mathcal{I}_k^j \left(0, \frac{\partial s}{\partial \nu_\Omega}, 0\right) \\
       &\quad \quad = \mathcal{H}_k \mathcal{T}\mathcal{P}_k \left( (I - \delta \mathcal{I}_k)^{-1} \left(0,\frac{\partial s}{\partial \nu_\Omega}, 0 \right) \right),
    \end{aligned}
\end{equation}
where, once again, the convergence of the infinite sum is understood in the sense of the Hilbert space $H^1_{loc}(\RR^2 \setminus \overline{D}) \times H^1(D).$ \\

\textbf{3.} In this step we define a candidate for the solution $u_\delta$ of \eqref{e.PDE}. Indeed, define
\begin{equation} \label{e.candidate}
    \hat{u}_\delta (x):= \left\{
    \begin{array}{cc}
        \hat{c}_\delta \psi_e + s + \delta \hat{\lambda}_\delta & x \in \RR^2 \setminus \overline{\Omega} \\
        \hat{c}_\delta + \delta \hat{\phi}_\delta & x \in \Omega \setminus \overline{D},\\
        \hat{c}_\delta \psi_d + \delta \hat{\chi}_\delta & x \in D.
    \end{array}
    \right.
\end{equation}
By construction, $\hat{u}_\delta$ and $\frac{1}{\e_\delta} \frac{\partial u_\delta}{\partial \nu} $ are continuous along $\partial \Omega$ and $\partial D,$ and in particular, $\hat{u}_\delta \in H^1_{loc}(\RR^2)$ by Lemma \ref{l.extbd}. Using partial sums, it is easily seen that by construction, $\hat{u}_\delta$ is a weak solution to \eqref{e.PDE}. As, by assumption, \eqref{e.PDE} is non-resonant, $\hat{u}_\delta$ coincides with its unique solution.\\

\textbf{4.} We are ready to complete the proof of the theorem. Defining $v_j := \lambda_j$ in the exterior, $:= e_j + \phi_j$ in the ENZ region, and $:= \chi_j$ in the dopant, the conclusion of the theorem follows using the estimates \eqref{e.ejest} and \eqref{e.jcorrest} from Step 2. The bound on $v_j$ follows by combining the estimates in Step 2 with \eqref{e.sbound}.
\end{proof}

Having completed the proof of our main result, we turn to the proofs of the corollaries stated in Section \ref{ss.main}.

\begin{proof}[Proof of Corollary \ref{c.leadord1}] From the proof of Theorem \ref{t.mt}, we know that $u_\delta = \hat{u}_\delta,$ defined in \eqref{e.candidate}. It follows, therefore, that
\begin{equation*}
    u_\delta(x) - v_\delta(x)  = \left\{ \begin{array}{cc}
      \left( \sum_{j=2}^\infty \delta^j e_j\right)\psi_e  + \sum_{j=2}^\infty \delta^j \lambda_j,   &  x \in \RR^2 \setminus \overline{\Omega}\\
         \sum_{j=2}^\infty \delta^j(e_j + \phi_j) & x \in \Omega \setminus \overline{D}\\
         \left( \sum_{j=2}^\infty \delta^j e_j\right) \psi_d + \sum_{j=2}^\infty \delta^j \chi_j,   &  x \in D.
    \end{array} \right.
\end{equation*}
In the above, each of the summations is interpreted as a convergent series in $H^1$ of the appropriate region (and $H^1_{loc}$ in the exterior). The proof of the corollary is then immediate from \eqref{e.highcorrrec}, $|\delta | < \frac{1}{\Lambda},$ and estimates \eqref{e.ejest} and \eqref{e.jcorrest}.
\end{proof}

\begin{proof}[Proof of Corollary \ref{c.idealfluid}]
Since the Poynting vector is defined by  $\vec{S}_\delta = \frac{1}{2} \vec{E}_\delta \times \overline{\vec{H}_\delta},$ the proof of this corollary is a short computation that combines  $\vec{H}_\delta = (0,0,u_\delta),$ and $\vec{E}_\delta = \frac{1}{i\omega}(-\partial_{x_2}\phi_0,\partial_{x_1}\phi_0, 0) + O(\delta)$  (see Remark \ref{r.electricfield}), along with \eqref{e.quadest} to pass to the limit $\delta \to 0$ limit. The boundary conditions for $S$ come from the Neumann boundary conditions for $\phi_0$ in \eqref{e.phizerodef}.
\end{proof}

\section{What if the Dopant is Resonant?} \label{s.exam}
For any choice of the dopant geometry, there will be some choices of $k$ for which the dopant is resonant, and it is
natural to ask what happens then. The answer one finds in \cite{Eng2} is that:
\begin{align} \label{e.phys}
   & \mbox{if the dopant resonance is excited then the magnetic field}\\[-3pt] \notag
  &  \mbox{vanishes in the ENZ region
and } \mu_{\rm eff} = \infty\,.
\end{align}
This section starts by briefly discussing what goes wrong when the dopant is resonant; then we shall recover \eqref{e.phys} using our framework, and we'll go beyond it by identifying the leading-order electric field.

The dopant is resonant when $k^2$ is an eigenvalue
of the Laplace operator on $D$ with zero Dirichlet boundary conditions, i.e., there exist finitely many (say $m \in \mathbb{N}$) eigenfunctions $\{U_j\}_{j=1}^m \subset H^1_0(D)$ that are orthonormal in $L^2(D),$ and satisfy, for every $j = 1, \ldots, m,$
\begin{eqnarray*}
    -\Delta U_j &=& k^2 U_j \quad \mbox{in $D$}\\
     U_j &=& 0 \quad \quad \mbox{at $\partial D$}.
\end{eqnarray*}
The finiteness of $m$ follows from standard theory. For such $k$, our framework breaks down in one of two
mutually exclusive ways:

\begin{enumerate}
\item
Suppose first that the eigenfunctions $U_j, j=1, \ldots, m$ each satisfy $\int_D U_j = 0,$ i.e.,
they each have mean zero in $D.$ In this case the auxiliary function $\psi_d$ exists but it isn't unique. Indeed,
setting $\Psi_d := \psi_d - 1,$ the existence of $\psi_d$ is equivalent to the solvability of the PDE
\begin{align*}
    -\Delta \Psi_d = k^2 (\Psi_d + 1) \quad \quad \mbox{ in } D,
\end{align*}
for $\Psi_d \in H^1_0(D).$ By the Fredholm alternative, $\Psi_d$ exists, as the constant function $1$ is
orthogonal in $L^2(D)$ to $U_j,$ i.e., $ \int_D U_j = 0$ for each $j = 1, \ldots, m. $ However it is nonunique, since for any $\{\alpha_j\}_{j=1}^m \subset \mathbb{C},$ the choices
\begin{align} \label{e.psid.formula}
    \psi_d = \Psi_d + 1 = 1 + k^2\left[\sum_{\lambda \in \sigma(-\Delta)\setminus \{k^2\}}
    \frac{\int_D \phi_\lambda \,dx }{\lambda - k^2} \phi_\lambda \right]+ \sum_{j=1}^m \alpha_j U_j
\end{align}
solve \eqref{e.PDEpsidop}. (Here $\sigma(-\Delta)$ denotes the set of Laplacian eigenvalues with zero
Dirichlet boundary data, and $\phi_\lambda$ is the eigenfunction associated with eigenvalue $\lambda,$
normalized so that $\|\phi_\lambda\|_{L^2(D)} = 1.$)

Heuristically, we can say in this case that ``the dopant resonance isn't excited,'' since there is apparently nothing
preventing us from choosing $\alpha_j = 0$ for $j=1,\ldots,m$. Moreover, since $\int U_j = 0,$ the value of
$c^*$ (which feels $\psi_d$ only via $\int_D \psi_d$) is oblivious to the choice of $\{\alpha_j \}$. However,
the correction $\chi_0$ (see \eqref{e.chi0defintro}) depends on the choice of $\{\alpha_j\}$.

Let us record that, from \eqref{e.psid.formula},
\begin{align} \label{e.psid.formula2}
    \int_D \psi_d = |D| + k^2 \left[\sum_{\lambda \in \sigma(-\Delta)\setminus \{k^2\}} \frac{\bigl( \int_D \phi_\lambda\,dx\bigr)^2}{\lambda - k^2}\right].
\end{align}
(This sum is convergent: by Weyl's law, the $j$th eigenvalue $\lambda_j \sim j^{2/d} = j,$ so the factors $\frac{1}{\lambda - k^2}$ share the decay of only the harmonic series [which logarithmically blows up!]. However, the convergence of the series defining $\int \psi_d$ is assisted by the fact that the sequence of $L^2$-normalized eigenfunctions $\{\phi_\lambda\}_\lambda$ converge weakly in $L^2$ to zero, so $\int_D\phi_\lambda\,dx \to 0.$)
\medskip

\item
Suppose instead that one of the eigenfunctions $U_j$ is such that $\int_D U_j \neq 0.$ (This happens, for instance,
if $k^2$ is the smallest Dirichlet eigenvalue of $-\Delta$ in $D$ since the eigenspace is then simple and the eigenfunction --
taken to be real-valued -- doesn't change sign.) In this case, rather than nonuniqueness the difficulty is nonexistence:
the boundary value problem \eqref{e.PDEpsidop} describing $\psi_d$ doesn't have a solution.

For reasons that will become clear presently, we view this as the case when ``the dopant resonance is excited.''
\end{enumerate}

To achieve a better understanding, it is natural to consider a limit in which the resonant situation is approached via
nonresonant ones. There are (at least) two natural approaches:

\begin{enumerate}
\item[(i)] We can ask what happens as $k^2$ approaches an eigenvalue along the real axis; this means choosing an eigenvalue
$\lambda_*$ of $-\Delta$ in $D$ with the Dirichlet boundary condition, and considering $k$ such that
$\lambda_* - k^2 =  \gamma,$ where $\gamma$ is real with $0< |\gamma| \ll 1.$
\item[(ii)] We can consider what happens when resonance is prevented by material losses. This can be achieved, for
example, by taking the magnetic permeability $\mu$ to have a positive imaginary part. This amounts in our framework
to taking $k^2 = \lambda_* + i \gamma$, where $\lambda_*$ is an eigenvalue (as above) and $0 < \gamma \ll 1$.
\end{enumerate}
Either way, our analysis applies to the limit $\delta \rightarrow 0$ when $\gamma$ is held fixed. We are therefore able to consider the limiting behavior when $\delta \to 0$ \textit{first, then} $\gamma \to 0.$

Let us pursue this idea in case (2) (when $\psi_d$ doesn't exist). Our analysis encompasses both alternatives (i) and (ii) suggested above, and is in fact more general. Fixing $\lambda_* \in \sigma(-\Delta),$ we consider \textit{any} $k \in \mathbb{C}$ with $0 \leqslant \arg \, k < \pi$ such that $k^2$ is near $\lambda_*$, say
$$
\lambda_* - k^2 = \gamma \quad \mbox{ with } |\gamma| \mbox{ small but nonzero}.
$$
We shall identify the limiting behavior as $|\gamma| \to 0$ with
$$
0 < |\delta| \ll |\gamma| \ll 1.
$$
To begin, we observe (by an obvious adjustment to \eqref{e.psid.formula2})  that
\begin{align} \label{e.psidintspec}
\int_D \psi_d = |D| + k^2 \left[\sum_{\lambda \in \sigma(-\Delta)} \frac{\left( \int_D \phi_\lambda \,dx\right)^2}{\lambda - k^2}\right].
\end{align}
We are assuming that at least one of the eigenfunctions associated with $\lambda = \lambda_*$ has nonzero mean. It follows that
the value of \eqref{e.psidintspec} blows up as $|\gamma| \to 0$, though it is finite (and large, of order
$1/|\gamma|$) for fixed $\gamma\neq 0$. Therefore
$$
|\beta| \sim \left| \int_{\partial D} \frac{\partial \psi_d}{\partial \nu_D}\,d \sh^1 \right| \sim  \frac{1}{|\gamma|},
$$
and it follows that $c^*$ is small, and $|\mu_{\rm eff}|$ is large:
\begin{equation} \label{e.c*gam}
|c^*| = O(|\gamma|) \quad \mbox{and} \quad \mu_{\rm eff} = O(|\gamma|^{-1})
\end{equation}
For the rest of this section it is convenient to track the dependence of various quantities on $\gamma,$ and we do this by a subscript: for instance, we specify the dependence of $c^*$ on $\gamma$ through the notation $c^*_\gamma.$

Theorem \ref{t.mt} applies, and asserts that the leading order corrections in the ENZ region, in the $\delta \to 0$ limit (for fixed $\gamma$) satisfy \eqref{e.phizerodef}:
\begin{equation} \label{e.phizerodef.gamma} \left\{
    \begin{aligned}
   & -\Delta \phi^\gamma_0 = k_\gamma^2 c_\gamma^* \quad \quad \mbox{ in } \Omega \setminus \overline{D} \\
    & \frac{\partial \phi_0^\gamma}{\partial \nu_\Omega} = c_\gamma^* \frac{\partial \psi_e}{\partial \nu_\Omega} + \frac{\partial s}{\partial \nu_\Omega} \quad \quad \mbox{ on } \partial \Omega \\
    &\frac{\partial \phi^\gamma_0}{\partial \nu_D} = c_\gamma^* \frac{\partial \psi^\gamma_d}{\partial \nu_D} \quad \quad \mbox{ on } \partial D,\\
    & \fint_{\Omega \setminus \overline{D}} \phi^\gamma_0 \,dx = 0.
    \end{aligned} \right.
\end{equation}
It is convenient to pass to the $\gamma \to 0^+$ limit by writing \eqref{e.phizerodef.gamma} in weak form. Toward this end, testing \eqref{e.phizerodef} with $h \in C^\infty(\Omega \setminus \overline{D}),$ we obtain
\begin{multline} \label{e.gammalim}
    \int_{\Omega \setminus \overline{D}} \nabla \phi_0^\gamma \cdot \nabla h\,dx - \int_{\partial \Omega} h \left(c_\gamma^* \frac{\partial \psi_e}{\partial \nu_\Omega} + \frac{\partial s}{\partial \nu_\Omega}  \right) \,d \sh^1 - \int_{\partial D} h\left(c_\gamma^* \frac{\partial \psi^\gamma_d}{\partial \nu_D} \right)\,d \sh^1 = \\ k_\gamma^2 c_\gamma^* \int_{\Omega \setminus \overline D} h \,dx .
\end{multline}
Arguing formally, if $\phi_0^\gamma$ converges weakly to $\widehat{\phi_0}$ in $H^1(D)$, then since $k_\gamma^2 \to \lambda_*$ , $\psi_e, s$ are independent of $\gamma,$ and $c^*_\gamma \to 0$ as $\gamma \to 0,$  it is clear how to pass to the limit in the first two terms on the left, as well as the term on the right hand side of \eqref{e.gammalim}; it remains to pass to the limit in the term $\int_{\partial D} h c_\gamma^* \frac{\partial \psi_d^\gamma}{\partial \nu_D}\,d \sh^1.$ The evaluation of this limit is hindered by the fact that $c_\gamma^* \to 0$ while, formally, $\psi_d^\gamma$ blows up. This can however be easily analyzed by integration by parts:
\begin{equation}\label{e.IBP}
    \begin{aligned}
     c_\gamma^* \int_{\partial D} h \frac{\partial \psi_d^\gamma}{\partial \nu_D}\,d\sh^1 &= c_\gamma^* \int_D h \Delta \psi_d^\gamma\,dx + c_\gamma^* \int_D \nabla h\cdot \nabla \psi_d^\gamma \,dx \\
    &= -c_\gamma^*k_\gamma^2 \int_D h  \psi_d^\gamma \,dx + c_\gamma^* \int_D \nabla h \cdot \nabla \psi_d^\gamma \,dx .
    \end{aligned}
\end{equation}
Let $\{U_j\}_{j=1}^m$ denote the eigenfunctions of $-\Delta$ in $D$ with homogeneous Dirichlet boundary conditions associated with an eigenvalue $\lambda_*$; here, $m$ denotes the multiplicity of $\lambda_*,$ which is finite by Fredholm theory. Then, using \eqref{e.psid.formula} with $ \lambda_* - k^2 = \gamma \not\in \sigma(-\Delta),$ one easily finds that
\begin{align*}
    \lim_{\gamma \to 0} \gamma \beta_\gamma = \lambda_*^2 \sum_{j=1}^m \left( \int_D U_j\,dx \right)^2,
\end{align*}
and so, the limit
\begin{align} \label{e.cbar}
    \lim_{\gamma \to 0} \frac{c^*_\gamma}{\gamma} = -\frac{\int_{\partial \Omega} \frac{\partial s}{\partial \nu_\Omega}\,d \sh^1}{\gamma \beta_\gamma} \to - \frac{\int_{\partial \Omega} \frac{\partial s}{\partial \nu_\Omega}\,d \sh^1 }{\lambda_*^2 \sum_{j=1}^m \left( \int_D U_j\,dx\right)^2} =: \overline{C}.
 \end{align}
exists and is finite. Since $k_\gamma^2 \to \lambda_*$ as $\gamma \to 0,$ using \eqref{e.cbar}, it then follows that the right hand side of \eqref{e.IBP} converges to
\begin{equation*}
    \begin{aligned}
     \overline{C}\sum_{j=1}^m -\lambda_*^2 \left( \int_D U_j\,dx   \right) \left(\int_D h U_j\,dx  \right) + \lambda_* \left( \int_D U_j\,dx   \right) \left( \int_D \nabla h \cdot \nabla U_j \,dx \right),
    \end{aligned}
\end{equation*}
which, by another integration by parts, rewrites as
\begin{align} \label{e.bdryterm}
    \overline{C}\lambda_* \sum_{j=1}^m \left(\int_D U_j \,dx \right) \left( \int_{\partial D} h \frac{\partial U_j}{\partial \nu_D}\,d \sh^1\right).
\end{align}
Therefore, sending $\gamma \to 0$ in \eqref{e.gammalim}, and using \eqref{e.bdryterm}, we find that $\widehat{\phi_0}$ is the unique solution to the problem
\begin{equation} \label{e.phihatzero} \left\{
    \begin{aligned}
   & -\Delta \widehat{\phi_0} = 0 \quad \quad \mbox{ in } \Omega \setminus \overline{D} \\
    & \frac{\partial \widehat{\phi_0}}{\partial \nu_\Omega}    = \frac{\partial s}{\partial \nu_\Omega} \quad \quad \mbox{ on } \partial \Omega \\
    &  \frac{\partial \widehat{\phi_0}}{\partial \nu_D} = \overline{C} \lambda_* \sum_{j=1}^m \left(\int_D U_j \,dx \right) \frac{\partial U_j}{\partial \nu_D} \quad \quad \mbox{ on } \partial D,\\
    & \fint_{\Omega \setminus \overline{D}} \widehat{\phi_0} \,dx = 0.
    \end{aligned} \right.
\end{equation}
In particular, the electric field distribution in the ENZ region, to leading order, is still identified by this procedure. (We note that \eqref{e.phihatzero} does have a solution, since by \eqref{e.cbar} the consistency condition for Laplace's equation with Neumann boundary conditions is satisfied.)

The preceding discussion applies only in the double limit when $\delta \to 0$ first then $\gamma \to 0$.
It would be interesting to have a more complete understanding, for example one that applies for $\delta$ in a neighborhood of
$0$ that doesn't depend on $\gamma$. However this seems to require a method different from that of the present paper.
\section*{Acknowledgements}
The authors warmly thank Nader Engheta for bringing this problem to our attention, and for helpful discussions. This work was partially supported by grants from the National Science Foundation (DMS-2009746, RVK) and the Simons Foundation (award \# 733694, RVK and RV). We are grateful to the anonymous referee for helpful comments on our paper.
\bibliographystyle{abbrv}
 \bibliography{ENZ}

\end{document}